\documentclass{amsart}

\usepackage{fullpage}
\usepackage{amsmath}
\usepackage{amsfonts}
\usepackage{amsthm}
\usepackage{amssymb}
\usepackage{latexsym}
\usepackage{dsfont}
\usepackage{enumitem}

\usepackage{bm,bbm}
\usepackage{url}

\usepackage{graphicx}
\usepackage{color}
\usepackage{tikz}
\usetikzlibrary{shapes,snakes,arrows,patterns}
\usepackage{float}

\usepackage{array}

\newcommand\R{\mathbb R}

\newcommand{\eps}{\varepsilon}

\newcommand{\xt}{\bold{x_T}}
\newcommand{\nn}{\bold{n}}
\newcommand{\afrak}{\mathfrak{a}}
\newcommand{\bfrak}{\mathfrak{b}}

\newcommand{\Qfrak}{\mathfrak{Q}}
\newcommand{\Rfrak}{\mathfrak{R}}
\newcommand{\Lfrak}{\mathfrak{L}}
\newcommand{\Scal}{\mathcal{S}}

\newtheorem{prop}{Proposition}

\newtheorem{remark}{Remark}

\definecolor{myblue}{rgb}{0.0,0.57,0.81}
\definecolor{myred}{rgb}{0.86,0,0.17}
\definecolor{mygreen}{rgb}{0,0.58,0}
\definecolor{mygray}{rgb}{0.3,0.3,0.3}
\definecolor{myyellow}{rgb}{1,0.84,0.024}
\definecolor{mycyan}{rgb}{0.19,0.835,0.87}
\definecolor{mypurple}{rgb}{0.635,0.055,0.67}
\definecolor{myorange}{rgb}{0.86,0.44,0.145}

\usepackage[normalem]{ulem}
\usepackage{marvosym}

\newcommand{\pfrak}{\mathfrak{p}}
\newcommand{\Ocal}{\mathcal{O}}

\newcommand{\bfu}{\bold{u}}
\newcommand{\bfv}{\bold{v}}
\newcommand{\bfq}{\boldsymbol{\phi}}

\newcommand{\drm}{\mathrm{d}}

\newcommand{\qfrak}{\mathfrak{q}}
\newcommand{\Acal}{\mathcal{A}}

\newcommand{\matE}{\underline{\mathcal{E}}}

\newcommand{\matEt}{\widetilde{\underline{\mathcal{E}}}}

\begin{document}

\title{Asymptotic analysis of a biphase tumor fluid flow. The weak coupling case.}
\author{Cristina Vaghi, Sebastien Benzekry, Clair Poignard\\
 Team MONC, Inria, Institut de Math\'ematiques de Bordeaux, CNRS, Bordeaux INP, Univ. Bordeaux, France}
\begin{abstract}
The aim of this paper is to investigate the asymptotic behavior of a biphase tumor fluid flow derived by 2-scale homogenisation techniques in recent works.
This biphase fluid flow model accounts for the capillary wall permeability, and the interstitial avascular phase, both being mixed in the limit homogenised problem. When the vessel walls become more permeable, we show that the biphase fluid flow exhibits a boundary layer that makes the computation of the full problem costly and unstable. In the limit, both capillary and interstitial pressures coincide except in the vicinity of the boundary where different boundary conditions are applied. Thanks to a rigorous asymptotic analysis, we prove that the solution to the full problem can be approached at any order of approximation by a monophasic model with appropriate boundary conditions on the tumor boundary and appropriate correcting terms near the boundary are given. Numerical simulations in spherical geometry illustrate the theoretical results.
\end{abstract}
\maketitle
\section{Introduction}
\subsection{Motivation}
Drug delivery in tumors is affected by the fluid flow phenomena that occur within the malignant tissues, which include blood flow, interstitial convection and transvascular transport \cite{baxter_transport_1989}. An efficient quantification of these processes is of great importance to evaluate the drug penetration within the tumor site.

Malignant tissues differ from normal tissues for several aspects \cite{chauhan_delivery_2011}. Neoplastic vasculature is unevenly distributed, leaving avascular spaces \cite{sevick_viscous_1989} and vessel walls are leaky and highly permeable \cite{dvorak_vascular_1995}. Furthermore, the tumor interstitial matrix is dense and heterogeneous \cite{jain_transport_1987}. These features lead to an elevated interstitial fluid pressure (IFP) at the center of the tumor with a sharp drop at the periphery, impacting the transport of fluids and drug \cite{boucher_microvascular_1992}.

Due to the high complexity of the architecture of tumors, two-scale asymptotic analysis has been employed to derive macroscopic models that take into account the microscopic properties of the malignant tissues \cite{shipley_multiscale_2010,penta_multiscale_2015}.
In a recent study, we derived several asymptotic models according to the magnitude of the hydraulic conductivity of the interstitium and of the vessel walls \cite{vaghi_macro-scale_2020}. Under the assumption of a periodic structure of neoplastic tissues, tumors are modeled as double porous media and Darcy's law describes the interstitial fluid flow and the blood transport. The coupling is driven by the permeability of the vessel walls. A biphase model was derived under the assumption of low permeability of the vessel walls, while a monophase model was obtained under the assumption of leaky vessel walls. The former consists of a coupled system of Darcy's equation and it is in agreement with previous findings \cite{shipley_multiscale_2010}. However, the numerical simulations of the biphase model are computationally expensive due to the coupling of two elliptic equations. Indeed, the domain discretization has to be thick enough to catch the sharp pressure gradient that occurs at the periphery of the tumor. 

Here, we prove that the biphase model can be approximated by the monophase model far from the boundary for large values of the hydraulic conductivity of the vessel walls. Moreover, boundary layer correctors can describe precisely the behavior of the biphase model solutions in a simple and computationally efficient way. 

Our methodology can be applied to efficiently simulate the fluid transport in tumors, which might give insights in the drug delivery process in malignant tissues and can be applied to optimize treatments.

\subsection{Model statement and objectives}

Given two Dirichlet boundary conditions $\pi_t$ and $\pi_c$ belonging to $H^{s+1/2}(\partial\Omega)$ with $s\geq0$ as large as necessary, we are interested in the analysis of the model derived by Chapman, Shipley {\it et al}~\cite{Shipley2010} also derived recently in \cite{vaghi_macro-scale_2020}. In this paper we focus on the case where the porosity tensors involved in~\cite{Shipley2010,vaghi_macro-scale_2020}  are colinear. After simple change of notation, the problem reads in the smooth domain $\Omega$ as
\begin{subequations}\label{eq:pbp}
\begin{align}
&-\nabla\cdot(\matE\nabla p^\eps_t)+\frac{\alpha^2}{\eps^2}(p^\eps_t-p^\eps_c)=0,\\
&-\nabla\cdot(\matE\nabla p^\eps_c)-\frac{\beta^2}{\eps^2}(p^\eps_t-p^\eps_c)=0,\intertext{with Dirichlet boundary conditions:}
&p^\eps_t|_{\partial \Omega}=\pi_t, \quad p^\eps_c|_{\partial\Omega}=\pi_c ,\qquad \text{on $\partial\Omega$.}
\end{align}
\end{subequations}
The tensor $\matE$ is a  positive definite  tensor satisfying
\begin{align}\matE_{\min}|X|^2\leq X^T\matE X\leq\matE_{\max}|X|^2,\qquad \forall X\in\R^3,\label{eq:coercivKE}
\end{align}
and $\alpha/\eps$ and $\beta/\eps$ are two positive constants that account for the permeability of the capillaries, the volume fraction of the capillary and the interstitium media, $\eps$ being a nondimension small parameter. 
\begin{remark}[The weakly coupled case]
Problem~\eqref{eq:pbp} is the case where the two phases are weakly coupled. Indeed, performing the change of unknowns 
$$q^\eps=(p^\eps_t-p^\eps_c)/2,\qquad p^\eps=(p^\eps_t+p^\eps_c)/2,$$
Problem~\eqref{eq:pbp} is equivalent to find $(p^\eps,q^\eps)$ such that 
\begin{subequations}\label{eq:pbpq}
\begin{align}
&-\nabla\cdot(\matE\nabla q^\eps)+\frac{\alpha^2+\beta^2}{\eps^2}q^\eps=0,\label{eq:pbqeps}
\\
&-\nabla\cdot(\matE\nabla p^\eps)=-\frac{\alpha^2-\beta^2}{\eps^2}q^\eps,\label{eq:pbpeps}\\
&q^\eps|_{\partial \Omega}=\frac{\pi_t-\pi_c}{2}, \quad p^\eps|_{\partial\Omega}=\frac{\pi_t+\pi_c}{2} ,\qquad \text{on $\partial\Omega$,}
\end{align}
\end{subequations}
hence $q^\eps$ is entirely determined by its boundary condition, while $p^\eps$ involves $q^\eps$.

To our opinion, the weakly couple case contains sufficient technical results, and describes already a lot of applications (for instance for scalar tensors  \cite{baxter_transport_1989}) to justify the present paper.
\end{remark}
We are interested in the asymptotic regime $\eps$ tending to 0, which corresponds to leaky vessel walls.
In particular, we show that $p^\eps_t-p^\eps_c$ decays exponentially fast from the domain boundary, making appear a typical skin depth effect on the pressure difference. In addition we show that  as $\eps$ goes 0, Problem~\eqref{eq:pbp} can be approached by the solution $\tilde{P}$ to the following monophase Laplace problem with well-designed boundary condition:
\begin{subequations}\label{eq:pbPtilde}
\begin{align}
&-\nabla\cdot(\matE\nabla \tilde{P})=0,
\\
&\tilde{P}|_{\partial\Omega}=\frac{\pi_t+\pi_c}{2}-\frac{\alpha^2-\beta^2}{\alpha^2+\beta^2} \frac{\pi_t-\pi_c}{2},\qquad \text{on $\partial\Omega$.}
\end{align}
\end{subequations}
The next section  is devoted to the prove the well-posedness of the problem, and the preliminary estimates of the solution $(q^\eps,p^\eps)$. In particular, we prove the exponential decay from the boundary of $q^\eps$. In section~\ref{sec:Asymp}, the asymptotic expansion of the solution at any order of $(q^\eps,p^\eps)$ and optimal error estimates are given, proving the efficacy of the methodology. Numerical simulations illustrate the theoretical results  in the last Section~\label{sec:Num}. The highly coupled case, where the porosity tensors are not colinear for  $q^\eps$ qnd $p^\eps$ will be treated in a forthcoming work. 

\section{Well-posedness and a priori estimates}

\begin{prop}\label{prop:estim0}Let $s\geq0$ and let $\pi_t$ and $\pi_c$ belongs to $H^{1/2+s}(\partial\Omega)$, there exists a unique solution $(p^\eps_t,p^\eps_c)$ to Problem \eqref{eq:pbp} in $(H^{1+s}(\Omega))^2$. In addition, there exists a constant $C$ independent of $\eps$ such that 
\begin{subequations}\label{eq:estim}
\begin{align}
\label{eq:estima}&\|p^\eps_t-p^\eps_c\|_{L^2(\Omega)}\leq C\left(\|\pi_t\|_{H^{1/2}(\partial\Omega)}+\|\pi_c\|_{H^{1/2}(\partial\Omega)}\right)\\
\label{eq:estimb}&\|p^\eps_t\|_{H^1(\Omega)}+\| p^\eps_c\|_{H^1(\Omega)}\leq \frac{C}{\eps}\left(\|\pi_t\|_{H^{1/2}(\partial\Omega)}+\|\pi_c\|_{H^{1/2}(\partial\Omega)}\right).
\end{align}
\end{subequations}
\end{prop}
\begin{proof}
The uniqueness of the solution is obvious and left to the reader. Thanks to standard elliptic regularity~\cite{EvansBook,GilbargBook}, one just has to prove the well-posedness in $(H^{1}(\Omega))^2$.
To prove the existence, denote by $f_t$ ({\it resp.} $f_c$) a lift of $\pi_t$ ($\pi_c$ {\it resp.}) which belongs to $H^{1+s}(\Omega)$ defined by 
\begin{align*}
&-\nabla\cdot(\matE\nabla f_t)=0,\quad \text{in $\Omega$}\\
&-\nabla\cdot(\matE\nabla f_c)=0,\quad \text{in $\Omega$}\\
&f_t|_{\partial\Omega}=\pi_t,\quad f_c|_{\partial\Omega}=\pi_c.
\end{align*}
It is well-known that there exists a constant $C$ such that 
\begin{align}
\label{estimEF}
\|f_t\|_{H^1(\Omega)}\leq C\|\pi_t\|_{H^{1/2}(\partial\Omega)},\qquad \|f_c\|_{H^1(\Omega)}\leq C\|\pi_c\|_{H^{1/2}(\partial\Omega)},
\end{align}

 Then, $(p^\eps_t,p^\eps_c)$ reads
\begin{align*}
(p^\eps_t,p^\eps_c)=(f_t+\phi^\eps_t,f_c+\phi^\eps_c),\end{align*}
where the couple $(\phi^\eps_t,\phi^\eps_c)$ satisfy
\begin{align*}
&-\nabla\cdot(\matE\nabla \phi^\eps_t)+\frac{\alpha^2}{\eps^2}(\phi^\eps_t-\phi^\eps_c)=-\frac{\alpha^2}{\eps^2}(f_t-f_c)\\
&-\nabla\cdot(\matE\nabla \phi^\eps_c)-\frac{\beta^2}{\eps^2}(\phi^\eps_t-\phi^\eps_c)=\frac{\beta^2}{\eps^2}(f_t-f_c),\intertext{with homogeneous Dirichlet condition:}
&\phi^\eps_t|_{\partial \Omega}=0=\phi^\eps_c|_{\partial\Omega},\qquad \text{on $\partial\Omega$,}
\end{align*}
Consider the following bilinear form $\Acal_\eps$ defined on $(H^1_0(\Omega))^2$ by:
\begin{align*}
\notag& \forall (\bfu,\bfv)=((u_1,u_2),(v_1,v_2))\in(H^1_0(\Omega))^2\times(H^1_0(\Omega))^2,
\\
&\Acal_\eps(\bfu,\bfv)=\beta^2\int_{\Omega} \matE\nabla u_1\nabla v_1 dx +\alpha^2\int_{\Omega} \matE\nabla u_2\nabla v_2 dx+\frac{\alpha^2\beta^2}{\eps^2}\int_{\Omega}(u_1-u_2)(v_1-v_2)dx.
\end{align*}

Thanks to Poincar\'e inequality, since $\matE$ is coercive by~\eqref{eq:coercivKE}, the bilinear form $\Acal$ is continuous and coercive on $(H^1_0(\Omega))^2$, and  the coercivity constant does not depend on $\eps$. Therefore there exists a unique solution $\bfq^\eps$ satisfying for any $\bfv\in(H^1_0(\Omega))^2$:
\begin{align}\label{eq:varf}
\Acal_\eps(\bfq^\eps,\bfv)=-\frac{\alpha^2\beta^2}{\eps^2}\int_{\Omega}(f_t-f_c)(v_1-v_2)dx.
\end{align}

To prove the {\it a priori} estimates, thanks to Dirichlet and Neumann trace theorems, one just has to show the inequalities on $\bfq^\eps$. Taking $\bfv=\bfq^\eps$ in \eqref{eq:varf}, one infers 
\begin{align*}
\begin{split}
\beta^2\matE_{\min}\|\nabla \phi^\eps_t\|^2_{L^2(\Omega)}+\alpha^2\matE_{\min}\|\nabla \phi^\eps_c\|^2_{L^2(\Omega)}+\frac{\alpha^2\beta^2}{\eps^2}\|\phi^\eps_t-\phi^\eps_c\|^2_{L^2(\Omega)}&\leq \frac{\alpha^2\beta^2}{\eps^2}\|f_t-f_c\|_{L^2(\Omega)} \|\phi^\eps_t-\phi^\eps_c\|_{L^2(\Omega)}.
\end{split}
\end{align*}
Then one infers successively that for a constant $C$ independent of $\eps$
$$\|\phi^\eps_t-\phi^\eps_c\|_{L^2(\Omega)}\leq C \left(\|\pi_t\|_{H^{1/2}(\partial\Omega)}+\|\pi_c\|_{H^{1/2}(\partial\Omega)}\right),$$
and
$$\|\nabla \phi^\eps_t\|_{L^2(\Omega)}+\|\nabla \phi^\eps_c\|_{L^2(\Omega)}
\leq \frac{C}{\eps}\left(\|\pi_t\|_{H^{1/2}(\partial\Omega)}+\|\pi_c\|_{H^{1/2}(\partial\Omega)}\right).$$

Estimates \eqref{eq:estima}--\eqref{eq:estimb} fall then easily thanks to \eqref{estimEF}  since $(p^\eps_t,p^\eps_c)=(f_t+\phi^\eps_t,f_c+\phi^\eps_c)$. \end{proof}
\begin{prop}\label{prop:expodec}
For any $d>0$, denote by $\Omega_d$ the inner domain defined by $\Omega_d=\left\{x\in\Omega:\, {\rm dist}(x,\partial\Omega)>d\right\}$.
There exists $\eps_0>$, $C_d$ and $\mu>0$ such that for any $\eps<\eps_0$,
\begin{align*}
\|p^\eps_t-p^\eps_c\|_{H^1(\Omega_d)}\leq C_de^{-\mu/\eps}\left(\|\pi_t\|_{H^{1/2}(\partial\Omega)}+\|\pi_c\|_{H^{1/2}(\partial\Omega)}\right).
\end{align*}
\end{prop}
\begin{proof}
The proof of the proposition is very similar to the proof of Haddar, Joly, Nguyen in~\cite{Haddar2005} even though the problem is slightly different. We  recall here the main ideas  for the self-consistency of the paper. 
Let $\phi$ be a smooth non negative function of $\Omega$ such that 
$$\phi(x)=\begin{cases}0,\quad \text{if  $x\in\Omega\setminus\Omega_{d/2}$},\\
2\mu,\quad \text{if  $x\in\Omega_d$},
\end{cases}
$$
where $\mu$ is a constant that is fixed later on.
Denote by $u^\eps$ the function of $\Omega$ defined by $(p^\eps_t-p^\eps_c)(x)=e^{-\phi(x)/\eps}u^\eps(x)$. It satisfies
\begin{align*}
\begin{split}
&-\nabla\cdot(\matE\nabla u^\eps)+\frac{1}{\eps}\left(\left(\matE+\matE^T\right)\nabla\phi\cdot\nabla u^\eps\right)+\frac{\alpha^2+\beta^2-\matE\nabla\phi\cdot\nabla\phi+\eps\nabla\cdot(\matE\nabla\phi)}{\eps^2}u^\eps=0,\quad\text{in $\Omega$},\\
&u^\eps|_{\partial\Omega}=\pi_t-\pi_c.
\end{split}
\end{align*}
Multiplying by $u^\eps$ and integrating by parts lead to 
\begin{align*}
\int_\Omega\matE\nabla u^\eps\cdot\nabla u^\eps dx+\frac{1}{\eps}\int_\Omega
\left(\left(\matE+\matE^T\right)\nabla\phi\cdot\nabla u^\eps\right)u^\eps dx&+\frac{1}{\eps^2}\int\left(\alpha^2+\beta^2-\matE\nabla\phi\cdot\nabla\phi+\eps\nabla\cdot(\matE\nabla\phi)\right){u^\eps}^2dx\\
&=\int_{\partial\Omega}u^\eps\partial_n u^\eps dx=\int_{\partial\Omega}(p^\eps_t-p^\eps_c)\partial_n (p^\eps_t-p^\eps_c) dx.\end{align*}
Then using the fact that $\phi$ vanishes identically near $\partial\Omega$, one infers
$$\int_\Omega
\left(\left(\matE+\matE^T\right)\nabla\phi\cdot\nabla u^\eps\right)u^\eps dx=-\frac{1}{2}\int_\Omega
\nabla\cdot\left(\left(\matE+\matE^T\right)\nabla\phi\right){u^\eps}^2dx,$$
and Proposition~\ref{prop:estim0} implies that
$$\left|\int_{\partial\Omega}u^\eps\partial_n u^\eps dx\right|=\left|\int_{\partial\Omega}(p^\eps_t-p^\eps_c)\partial_n (p^\eps_t-p^\eps_c) dx\right|\leq\frac{C^2}{\eps^2}\left(\|\pi_t\|_{H^{1/2}(\partial\Omega)}+\|\pi_c\|_{H^{1/2}(\partial\Omega)}\right)^2.$$
Thus one infers the following estimate
\begin{align*}
\int_\Omega\matE\nabla u^\eps\cdot\nabla u^\eps dx+&\frac{1}{\eps^2}\left|\int\left(\alpha^2+\beta^2-\matE\nabla\phi\cdot\nabla\phi+\frac{\eps}{2}\nabla\cdot((\matE-\matE^T)\nabla\phi)\right){u^\eps}^2dx\right|\\
&\leq\frac{C^2}{\eps^2}\left(\|\pi_t\|_{H^{1/2}(\partial\Omega)}+\|\pi_c\|_{H^{1/2}(\partial\Omega)}\right)^2.
\end{align*}
One then just has to choose $\mu$ independent of $\eps$ such that $$\|\matE\nabla\phi\cdot\nabla\phi\|\leq (\alpha^2+\beta^2)/3,$$
to infer that for $\eps$ small enough
$$\left\|(\alpha^2+\beta^2)/2-\matE\nabla\phi.\nabla\phi+\frac{\eps}{2}\nabla\cdot((\matE-\matE^T)\nabla\phi\right\|_{L^{\infty}(\Omega)}\geq \frac{\alpha^2+\beta^2}{2}.$$
Then \begin{align*}\|p^\eps_t-p^\eps_c\|_{H^1(\Omega_d)}=e^{-2\mu/\eps}\|u^\eps\|_{H^1(\Omega)}&\leq \frac{Ce^{-2\mu/\eps}}{\eps}\left(\|\pi_t\|_{H^{1/2}(\partial\Omega)}+\|\pi_c\|_{H^{1/2}(\partial\Omega)}\right)\\&\leq \frac{C}{e\mu}e^{-\mu/\eps}\left(\|\pi_t\|_{H^{1/2}(\partial\Omega)}+\|\pi_c\|_{H^{1/2}(\partial\Omega)}\right).\end{align*}

\end{proof}
\section{Asymptotic analysis}\label{sec:Asymp}
Consider here Problem~\eqref{eq:pbpq} satisfied by $q^\eps=(p^\eps_t-p^\eps_c)/2$ and by $p^\eps=(p^\eps_t+p^\eps_c)/2$:
\begin{align*}&-\nabla\cdot(\matE\nabla q^\eps)+\frac{\alpha^2+\beta^2}{\eps^2}q^\eps=0,\qquad\text{in $\Omega$},\\
&-\nabla\cdot(\matE\nabla p^\eps)+\frac{\alpha^2-\beta^2}{\eps^2}q^\eps=0,\qquad\text{in $\Omega$},
\intertext{with Dirichlet boundary conditions:}
&p^\eps|_{\partial \Omega}=(\pi_t+\pi_c)/2, \quad q^\eps|_{\partial\Omega}=(\pi_t-\pi_c)/2 ,\qquad \text{on $\partial\Omega$,}
\end{align*}
\subsection{Geometry}  Let $\drm>0$ be a fixed distance to $\partial\Omega$, such that the tubular neighborhood of $\partial\Omega$ of radius $\drm$ is parameterized by local coordinates. More precisely, let
$\xt=(x_1,x_2)$    be    a   system    of    local   coordinates    on
$\partial\Omega=\left\{\Psi(\xt)\right\}.$
Define the map $\Phi$ by
\begin{align}
\forall (\xt,x_3)\in\partial\Omega\times\R,\quad \Phi(\xt,x_3)=\Psi(\xt)-x_3 \nn(\xt),\label{eq:mapPhi}
\end{align}
where   $\nn$   is   the   normal   vector  of   $\partial\Omega$ outwardly directed.    Then we assume that   $\Ocal_{\drm}$ the tubular neighborhood of $\partial\Omega$ is
parameterized as
$$
\Ocal_{\drm}=\left\{\Phi(\xt,x_3),\quad
  (\xt,x_3)\in\partial\Omega\times(0,\drm)\right\}.
$$
The  Euclidean  metric  in   $(\xt,x_3)$  is  given  by  the  $3\times
3$--matrix                 $(g_{ij})_{i,j=1,2,3}$                where
$g_{ij}=\langle\partial_{i}\Phi,\partial_{j} \Phi\rangle$:
\begin{subequations}
\label{coeffg}
  \begin{align}
  &g_{33}=1,\\
    \forall \alpha\in\{1,2\},\quad  &g_{\alpha 3} = g_{3
      \alpha}    =   0,\\
    \forall(\alpha,\beta)\in\{1,2\}^2, \quad
    &g_{\alpha\beta}(\xt,x_3)=g^0_{\alpha\beta}(\xt)-2x_3b_{\alpha\beta}(\xt)+x_3^2
    c_{\alpha\beta}(\xt), \label{g-alphabeta}%
    \intertext{where} &g^0_{\alpha\beta}=\langle\partial_\alpha
    \Psi,\partial_\beta \Psi\rangle,\quad
    b_{\alpha\beta}=\langle\partial_\alpha \nn,\partial_\beta
    \Psi\rangle,\quad c_{\alpha\beta}=\langle\partial_\alpha
    \nn,\partial_\beta \nn\rangle.
  \end{align}\end{subequations}

\subsection{The operator $\nabla\cdot(\matE\nabla\cdot)$ in local coordinates}

Define  $\matEt$ the matrix $\matE$ written in the new basis $(\partial_i\Phi)_{i=1,2,3}$. In other words, 
\begin{subequations}\label{matEtilde}\begin{align}
 \matEt=P^{-1}\matE P,\intertext{where $P$ is the transfer matrix from the Euclidean basis to $(\partial_i\Phi)_{i=1,2,3}$:}
P=\begin{pmatrix}\partial_1\Phi,\partial_2\Phi,\partial_3\Phi\end{pmatrix}.
\end{align}
\end{subequations}

We denote by  $(g^{ij})$ the inverse matrix of  $(g_{ij})$ defined by \eqref{coeffg}, and by $g$
the determinant of $(g_{ij})$.  For any integer $l\geq0$ define the following geometric coefficients independent of $\eps$:
\begin{align}
  \begin{cases}
    &
    \afrak^l_{ij}=(-1)^l\left.\partial_3^l\left(\dfrac{\partial_i\left(\sqrt{g}g^{ij}
          \right)}{\sqrt{g}}\right)\right|_{x_3=0},
    \quad \text{for $(i,j)\in\{1,2,3\}^2$},\\
    & \bfrak^l_{\alpha\beta}=(-1)^l\left.\partial_3^l\left(g^{\alpha\beta}\right)
    \right|_{x_3=0},\quad \text{for $(\alpha,\beta)\in\{1,2\}^2$},
  \end{cases} \label{aijAl}
\end{align}
and we denote by  $\Scal^l_{\partial\Omega,\matEt}$ the sequence of differential operators  on $\partial\Omega$ of
order 2 defined by
\begin{align}
   \Scal^l_{\partial\Omega,\matEt}=\sum_{\alpha,\beta=1,2}\matEt_{\alpha\beta}
  \left( \afrak^l_{\alpha\beta}\partial_\beta+\bfrak^l_{\alpha\beta}\partial_{\alpha}\partial_\beta\right).
  \label{Scal}
\end{align}

Considering the rescaled local coordinates $(\xt,\rho)=(\xt,x_3/\eps)$, the  operator $\nabla\cdot(\matE\nabla\cdot)$ can be expanded in series of $\eps$ as follows:
\begin{align}
\notag
 \nabla\cdot(\matE\nabla\cdot)&=\frac{1}{\sqrt{g}}
  \sum_{i,j=1,2,3}\partial_{i}\left(\sqrt{g}\matEt_{ij}g^{ij}\partial_{j}\right),\qquad&\forall(\xt,x_3)\in\partial\Omega\times(0,\drm),&&&
 \\
\label{divEgrad}
 &=\frac{\matEt_{33}}{\eps^2}\partial^2_\rho-\frac{\matEt_{33}}{\eps}\afrak^0_{33}(\xt)
 \partial_\rho
+\sum_{\ell\geq0}(-1)^\ell\eps^\ell\frac{\rho^\ell}{\ell!}\Lfrak_{\matEt,\ell},\qquad&
 \forall(\xt,\rho)\in\partial\Omega\times(0,\drm/\eps),&&&
 \end{align}
{where the operators $\Lfrak_{\matEt,\ell}$ are of first order in $\rho$ and of second order in $\xt$, and given for any $\ell\geq0$ by}
\begin{align*}
 \forall l\geq 0,\quad
 \Lfrak_{\matEt,\ell}=\left(\frac{\rho}{\ell+1}{\matEt}_{33}\afrak^{\ell+1}_{33}(\xt)
    \partial_\rho+\Scal^\ell_{\partial\Omega,\matEt}\right),\qquad \forall(\xt,\rho)\in\partial\Omega\times(0,\drm/\eps) .
\end{align*}
We refer to \cite{Perrussel2013} for the detailed calculation of these expansions in the case of Laplace operator. 

%
%

\subsection{Asymptotic expansion of $q^\eps$}

The problem satisfied by $q^\eps$ is similar to the Helmholtz problem in high conductive materials that has been studied extensively in the last decade~\cite{Haddar2005,Dauge2010,Caloz2010,Caloz2011, Dauge2014}, in different context.
We recall here the main results regarding the expansion of $q^\eps$, and the proof is given in the Appendix section~ref{sec:appendix}
 for self consistency of the paper. 

It is important to note that thanks to the properties of  $\matE$ given in~\eqref{eq:coercivKE}, $\matEt_{33}$ is strictly positive.

Denote by $\gamma$ the positive parameter such that
 $$\gamma^2=\frac{\alpha^2 +\beta^2}{\matEt_{33}}, \qquad \gamma>0.$$
The following proposition provides the asymptotic expansion of $q^\eps$:
\begin{prop}\label{prop:estimexpan} Let $\drm>0$, such that the mapping $\Phi$ defined by \eqref{eq:mapPhi} is smooth and one-to-one. Define the smooth cut-off function $\chi_\drm$ equal to 1 in $\Omega\setminus\overline{\Omega_{\drm}}$, whose support is compactly embedded in $\Omega\setminus\overline{\Omega_{2d}}$.  

Assume that $\Omega$ is a smooth domain, and that $\pi_t$ and $\pi_c$ are smooth functions of $\partial\Omega$.

Then for any $k\geq0$, there exists $C_k$ depending on $\pi_t$, $\pi_c$ and their derivatives such that 
 \begin{align*}
\| q^\eps-\sum_{\ell=0}^k\chi_d\eps^\ell\qfrak^\ell(\xt,x_3/\eps)\|_{L^2(\Omega)}\leq C_k\eps^{k+1},
\\
\| q^\eps-\sum_{\ell=0}^k\chi_d\eps^\ell\qfrak^\ell(\xt,x_3/\eps)\|_{H^1(\Omega)}\leq C_k\eps^{k+1/2},
\end{align*}
where the profiles $\qfrak^\ell$ satisfy:
\\
\noindent $\bullet$ For $k=0$:
 \begin{align} \label{eq:k=0}
 \begin{split}
& -\partial^2_\rho \qfrak^0+\gamma^2\qfrak^0=0,\qquad \forall(\xt,\rho)\in\partial\Omega\times(0,+\infty),\\
&\qfrak^0|_{\eta=0}=\frac{1}{2}(\pi_t-\pi_c),\qquad \qfrak^0\rightarrow_{\rho\rightarrow +\infty}=0,
\end{split}
\end{align}
$\bullet$ For $k\geq 1$:
 \begin{align} \label{eq:k}
 \begin{split}
& -\partial^2_\rho \qfrak^k+\gamma^2\qfrak^k=-\frac{\afrak^0_{33}(\xt)}{\matEt_{33}}
 \partial_\rho \qfrak^{k-1}
+\sum_{\ell=0}^{k-2}\frac{(-1)^\ell}{\matEt_{33}}\frac{\rho^\ell}{\ell!}\Lfrak_{\matEt,\ell}(\qfrak^{k-2-\ell}),\qquad \forall(\xt,\rho)\in\partial\Omega\times(0,+\infty), \\
&\qfrak^k|_{\eta=0}=0,\qquad \qfrak^k\rightarrow_{\rho\rightarrow +\infty}=0,
\end{split}
\end{align}
where by convention, $\qfrak^n\equiv0$ for $n\leq 0$.

More precisely, 
\begin{subequations}\label{eq:qfrakk}
 \begin{align}\qfrak^0(\xt,\rho)=\frac{1}{2}(\pi_t-\pi_c)e^{-\gamma\rho},\intertext{and for any $k\geq0$, }
 \qfrak^k(\xt,\rho)= \Qfrak_{k}(\xt,\rho)e^{-\gamma\rho},
 \end{align}
 \end{subequations}
 where $\Qfrak_{k}$ is polynomial of degree $k$ in the variable $\rho$, which vanishes in $\rho=0$ and whose coefficients are smooth functions of  $\partial\Omega$.
\end{prop}
\begin{proof}
Problem~\eqref{eq:pbqeps} reads in local coordinates in the vicinity of $\partial\Omega$ as follows:
\begin{subequations}\label{eq:q}\begin{align}
&\notag\forall(\xt,\rho)\in\partial\Omega\times(0,\drm/\eps),\\
&-\partial^2_\rho q^\eps+\gamma^2q^\eps+\eps \frac{\afrak^0_{33}(\xt)}{\matEt_{33}}
 \partial_\rho q^\eps
-\sum_{\ell\geq0}\frac{(-1)^\ell}{\matEt_{33}}\eps^{\ell+2}\dfrac{\rho^\ell}{\ell!}\Lfrak_{\matEt,\ell}q^\eps=0,\\
&q^\eps|_{\rho=0}=\frac{1}{2}(\pi_t-\pi_c).
 \end{align}
  \end{subequations}
 
\noindent$\bullet$~{Formal expansion.}

  Set the Ansatz 
 $$q^\eps(x)=\chi_d(x)\sum_{k\geq0} \eps^k \qfrak^k\circ\Phi^{-1}(x),\quad\forall x\in \Omega,$$
where $\qfrak^k$ are defined in $\partial\Omega\times(0,+\infty)$.

By Proposition~\ref{prop:expodec}, the exponential decay to 0 of $q^\eps$ implies that the above ansatz is consistent in $\Omega_\drm$.

 Identifying the terms with the same power in $\eps^k$ implies that the coefficients of the expansions $\qfrak^k$ satisfy the following inductive elementary problems.
 
\noindent $\bullet$ For $k=0$:
 \begin{align*} 
 \begin{split}
& -\partial^2_\rho \qfrak^0+\gamma^2\qfrak^0=0,\qquad \forall(\xt,\rho)\in\partial\Omega\times(0,+\infty),\\
&\qfrak^0|_{\eta=0}=\frac{1}{2}(\pi_t-\pi_c),\qquad \qfrak^0\rightarrow_{\rho\rightarrow +\infty}=0,
\end{split}
\end{align*}
$\bullet$ For $k\geq 1$:
 \begin{align*} 
 \begin{split}
& -\partial^2_\rho \qfrak^k+\gamma^2\qfrak^k=-\frac{\afrak^0_{33}(\xt)}{\matEt_{33}}
 \partial_\rho \qfrak^{k-1}
+\sum_{\ell=0}^{k-2}\frac{(-1)^\ell}{\matEt_{33}}\frac{\rho^\ell}{\ell!}\Lfrak_{\matEt,\ell}(\qfrak^{k-2-\ell}),\qquad \forall(\xt,\rho)\in\partial\Omega\times(0,+\infty), \\
&\qfrak^k|_{\eta=0}=0,\qquad \qfrak^k\rightarrow_{\rho\rightarrow +\infty}=0,
\end{split}
\end{align*}
where by convention, $\qfrak^n\equiv0$ for $n\leq 0$.
The expression of  $\qfrak^0$ is obvious since it satisfies the equation \eqref{eq:k=0}.
 Assuming that equality~\eqref{eq:qfrakk} is satisfied up to a give $(k-1)\geq0$. Then $\qfrak^k$ satisfies
 \begin{align}\label{eq:HK}
 \begin{split}  -\partial^2_\rho \qfrak^k+\gamma^2\qfrak^k&=-\frac{\afrak^0_{33}(\xt)}{\matEt_{33}}
 \left(\partial_\rho \Qfrak_{k-1}-\gamma \Qfrak_{k-1}\right)e^{-\gamma\rho}
\\&+e^{-\gamma\rho}\sum_{\ell=0}^{k-2}\frac{(-1)^\ell}{{\matEt_{33}}}\frac{\rho^\ell}{\ell!}\left(\frac{\rho}{\ell+1}\afrak^{\ell+1}_{33}
    \left(\partial_\rho \Qfrak_{k-2-\ell}-\gamma \Qfrak_{k-2-\ell}\right)+\Scal^\ell_{\partial\Omega,\matEt}(\Qfrak_{k-2-\ell})\right).\end{split}
\end{align} 
One then just has to observe that the righthand side of the above equality reads as $\Rfrak_{k-1}e^{-\gamma\rho}$, where $\Rfrak_{k-1}=\sum_{\ell=0}^{k-1}b_{\ell}(\xt)\rho^\ell$ is polynomial in $\rho$ and with smooth coefficients in $\partial\Omega$ by induction hypothesis. Denoting by $(c_\ell)_{\ell=1}^k$ the sequence of smooth functions defined on $\partial\Omega$ by
\begin{align*}
&c_{k}=\frac{b_{k-1}}{2\gamma k},
\intertext{for $\ell=k-1,\cdots,1$}
&c_{\ell}=\frac{1}{2\gamma \ell}\left((\ell+1)\ell c_{\ell+1}+b_{\ell}\right),
\end{align*}
then $\qfrak^k=\sum_{\ell=1}^{k}c_\ell(\xt)\rho^\ell e^{-\gamma\rho}$ satisfies \eqref{eq:HK} and vanishes at the boundaries $\rho=0$ and $\rho\rightarrow+\infty$. 

\noindent$\bullet$~{Proof of the estimates.}

The proof of the estimates is standard and we just recall here the main ingredients. The reader can refer to \cite{Haddar2005} for more details.
Thanks to Proposition~\ref{prop:expodec},  since  $\|q^\eps\|_{H^1(\Omega_\drm)}\leq e^{-\mu/\eps}$ for a specific $\mu>0$, and since 
$\sum_{\ell=0}^k\chi_d\eps^\ell\qfrak^\ell\circ\Phi^{-1}$ also decays exponentially fast in $\Omega_\drm$
one just has to prove the estimate in the tubular neighborhood $\Ocal_\drm$ of $\partial\Omega$, where local coordinates can be used.
Using \eqref{divEgrad}, simple a priori estimates show that 
$$\| q^\eps-\sum_{\ell=0}^n\chi_d\eps^\ell\qfrak^\ell(.,./\eps)\circ\Phi^{-1}\|_{H^1(\Omega)}\leq C_n\eps^{n-1/2},$$
for any $n$. Then the proposition is proven by observing that since $\qfrak^{k+1}P_{k+1}e^{-\gamma\rho}$, with $P_{k+1}$ polynomial in $\rho$ which vanishes in 0, one the following estimate 
\begin{align*}
\| q^\eps-\sum_{\ell=0}^k\chi_d\eps^\ell\qfrak^\ell(.,./\eps)\circ\Phi^{-1}\|_{H^1(\Omega)}&\leq\| q^\eps-\sum_{\ell=0}^{k+2}\chi_d\eps^\ell\qfrak^\ell(.,./\eps)\circ\Phi^{-1}\|_{H^1(\Omega)}+\eps^{k+1}\| \chi_d(\qfrak^{k+1}+\eps\qfrak^{k+2})\|_{H^1(\Omega)}\\&\leq C_k(\eps^{k+3/2}+\eps^{k+1/2}).\end{align*}
\end{proof}
\subsection{Asymptotic expansion of $p^\eps=(p^\eps_t+p^\eps_c)/2$}
We are now ready to approximate the function $p^\eps=(p^\eps_t+p^\eps_c)/2$, which is coupled to $q^\eps$ by Problem~\eqref{eq:pbpeps}.
Using the expansion of $q^\eps$  given by Proposition~\eqref{prop:estimexpan},  $p^\eps$ satisfies formally
\begin{subequations}\label{eq:pbpepsk}
\begin{align}
&-\nabla\cdot(\matE\nabla p^\eps)=-\chi_d\frac{\alpha^2-\beta^2}{\eps^2} \sum_{k\geq 0}\eps^k\qfrak^k(\xt, x_3/\eps),\qquad \text{in $\Omega$},\\
&p^\eps|_{\partial \Omega}=(\pi_t+\pi_c)/2,
\end{align}
\end{subequations}
where by abuse of notation the mapping $\Phi$ is omitted,  and where $\delta_0^k$ is the Kronecker symbol equal to 1 is $k=0$ and 0 elsewhere.

Since the each term  $\qfrak_k$ decays exponentially fast as $x_3/\eps$, a fine mesh is necessary to solve the above equality in order to capture the the source term $q^\eps$.
To prevent this drawback, we propose to determine an asymptotic expansion of $p^\eps$ which splits between the fast variable $x_3/\eps$
 and the slow variable . More precisely, we  look for 
 $$p^\eps=\sum_{k\geq 0}\eps^k (P_k+\chi_d(x_3)\pfrak_k(\xt,x_3/\eps)),$$
where  each function $\pfrak_k$ describes the behavior near the boundary while $P_k$ is a function depending on the slow-variable far from the boundary. 
As before, the idea is to use the expansion of the operator $\nabla\cdot(\matE\nabla\cdot)$ in the local coordinates
 \subsubsection{Expansion of $p^\eps$}
Injecting the Ansatz in Problem~\eqref{eq:pbpeps} we infer that for any $k\geq 0$ the profiles $(\pfrak_k)_{k\geq0}$ are defined on  $\Gamma\times(0,+\infty)$  vanish as $\rho$ goes to infinity and  satisfy
\begin{subequations} \label{eq:pfrakk}\begin{align}
&\partial^2_\rho\pfrak_0=\frac{\alpha^2-\beta^2}{\matEt_{33}} \qfrak_{0},
\intertext{and for $k\geq 1$, }
&\partial^2_\rho \pfrak_k= \frac{\alpha^2-\beta^2}{\matEt_{33}} \qfrak_{k}+\frac{\afrak^0_{33}(\xt)}{\matEt_{33}}
 \partial_\rho \pfrak_{k-1}
-\sum_{n=0}^{k-2}\frac{(-1)^n}{\matEt_{33}}\dfrac{\rho^n}{n!}\Lfrak_{\matEt,n}\pfrak_{k-n-2}.
 \end{align}
 \end{subequations}
  \begin{prop}\label{prop:pfrakk}
 For any $k\geq0$, there exists a unique function $\pfrak_k$ defined on $\Gamma\times\R^+$, vanishing as $\rho$ tends to infinity and satisfying Problem~\eqref{eq:pfrakk}. Moreover $\pfrak_k$ is  given by 
 $$\pfrak_k=\frac{\alpha^2-\beta^2}{\alpha^2+\beta^2}\qfrak_k.$$
More precisely, 
 \begin{align*}
 \pfrak^0(\xt,\rho)=\frac{\alpha^2-\beta^2}{\alpha^2+\beta^2}\frac{\pi_t-\pi_c}{2}e^{-\gamma\rho},
 \intertext{and for any $k\geq0$, }
 \pfrak^k(\xt,\rho)=\frac{\alpha^2-\beta^2}{\alpha^2+\beta^2} \Qfrak_{k}(\xt,\rho)e^{-\gamma\rho},
 \end{align*}
 where $\Qfrak_{k}$ is polynomial of degree $k$ in the variable $\rho$, which vanishes in $\rho=0$ and whose coefficients are smooth functions of  $\partial\Omega$.
In particular, 
$$\pfrak_k|_{\rho=0}=\delta^k_0\frac{\alpha^2-\beta^2}{\alpha^2+\beta^2}\frac{\pi_t-\pi_c}{2}.$$
 \end{prop}
 \begin{proof}
Uniqueness easily comes from the fact that 2 functions satisfying \eqref{eq:pfrakk} are equal modulo an affine function, which is necessary zero since it vanishes as $\rho$ goes to infinity.

To prove that $\pfrak_k$ equals $(\alpha^2-\beta^2)\qfrak_k/\gamma^2$, one just has to use induction. Indeed it is obviously true for $k=0$, by the first equation of \eqref{eq:pfrakk}. Assume that this is true up to the rank $k-1\geq0$. One has, by definition of $\qfrak_k$:
\begin{align*}
\frac{\alpha^2-\beta^2}{\alpha^2+\beta^2}\partial^2_\rho \qfrak_k&= \frac{\alpha^2-\beta^2}{\matEt_{33}} \qfrak_{k}+\frac{\alpha^2-\beta^2}{\alpha^2+\beta^2}\frac{\afrak^0_{33}(\xt)}{\matEt_{33}}
 \partial_\rho \qfrak_{k-1}
-\frac{\alpha^2-\beta^2}{\alpha^2+\beta^2}\sum_{n=0}^{k-2}\frac{(-1)^n}{\matEt_{33}}\dfrac{\rho^n}{n!}\Lfrak_{\matEt,n}\qfrak_{k-n-2}\\
&=\frac{\alpha^2-\beta^2}{\matEt_{33}} \qfrak_{k}+\frac{\afrak^0_{33}(\xt)}{\matEt_{33}}
 \partial_\rho \pfrak_{k-1}
-\sum_{n=0}^{k-2}\frac{(-1)^n}{\matEt_{33}}\dfrac{\rho^n}{n!}\Lfrak_{\matEt,n}\pfrak_{k-n-2},
\end{align*}
by induction. Thus $\dfrac{\alpha^2-\beta^2}{\alpha^2+\beta^2}\qfrak_k$ and $\pfrak$ satisfy the same problem and both functions vanish at infinity, so they are equal, which ends the proof.
 \end{proof}
 Then  $P_0$ is given by
  \begin{align*}
&-\nabla\cdot(\matE\nabla P_0)=0,\qquad \text{in $\Omega$},\\
&P_0|_{\partial \Omega}=\left(\frac{\pi_t+\pi_c}{2}\right)-\frac{\alpha^2-\beta^2}{\alpha^2+\beta^2}\frac{\pi_t-\pi_c}{2},
\end{align*}
which is nothing but the problem~\eqref{eq:pbPtilde} satisfied by $\tilde{P}$,
and for any $k\geq1$, $P_k=0$.
One has the following proposition.
\begin{prop}\label{prop:peps} 
Assume that $\Omega$ is a smooth domain, and that $\pi_t$ and $\pi_c$ are smooth functions of $\partial\Omega$.

Then for any $k\geq0$, there exists $C_k$ depending on $\pi_t$, $\pi_c$ and their derivatives such that 
 \begin{align*}
 &\| p^\eps-P_0-\sum_{n=0}^k\eps^n\pfrak_n(\xt,x_3/\eps)\|_{L^2(\Omega)}\leq C_k\eps^{k+1},\\
&\| p^\eps-P_0-\sum_{n=0}^k\eps^n\pfrak_n(\xt,x_3/\eps)\|_{H^1(\Omega)}\leq C_k\eps^{k+1/2}
\end{align*}
\end{prop}
\begin{proof}
The proof is standard and is performed recursively.
Using Proposition~\ref{prop:estimexpan},  the expansion of $\nabla\cdot(\matE\nabla\cdot)$, and by construction for any $k\geq 0$ standard energy estimates show that
$$\| p^\eps-P_0-\sum_{n=0}^k\eps^n\pfrak_n(\xt,x_3/\eps)\|_{H^1(\Omega)}\leq C_k\eps^{k-1},$$
then for a given $k$, one also has easily
\begin{align*} 
p^\eps-P_0-\sum_{n=0}^k\eps^n\pfrak_n(\xt,x_3/\eps)&= p^\eps-P_0-\sum_{n=0}^{k+3}\eps^n\pfrak_n(\xt,x_3/\eps)+\eps^k \sum_{n=1}^{3}\eps^n\pfrak_{k+n}(\xt,x_3/\eps).
\end{align*}
By construction, for any $n$ there exists a constant $C_{n}$ which depends on $n$, and $\pi_t, \pi_c$ and their derivatives such that
$$ \|\chi_\drm\pfrak_n(\xt,x_3/\eps)\|_{L^2(\Omega)}\leq C_{n}\sqrt{\eps},\quad \|\chi_\drm\pfrak_n(\xt,x_3/\eps)\|_{H^1(\Omega)}\leq \frac{C_{n}}{\sqrt{\eps}},$$
which end the proof.
\end{proof}

\section{Applications and numerical simulations}\label{sec:Num}
\subsection{Expression of the first order of approximation}
Even though the above Propositions~\ref{prop:estimexpan}--\ref{prop:peps} makes it possible to derive the whole expansion of $q^\eps$ and $p^\eps$, in the applications only the first order terms are often sufficient.
Using \eqref{eq:k=0}--\eqref{eq:k}, one infers successively
\begin{subequations}\label{eq:firstorder}
\begin{align}
&\qfrak^0(\xt,\rho)=\frac{(\pi_t-\pi_c)(\xt)}{2}e^{-\gamma\rho},\label{eq:qfrak0}\\
&\qfrak^1(\xt,\rho)=\frac{(\pi_t-\pi_c)(\xt)}{2}\frac{\afrak^0_{33}(\xt)}{2}\rho e^{-\gamma\rho},\label{eq:qfrak1}
 \end{align}
and thus Proposition~\ref{prop:pfrakk}, one infers
 \begin{align}
&\pfrak_0(\xt,\rho)=\frac{\alpha^2-\beta^2}{\gamma^2}\frac{(\pi_t-\pi_c)(\xt)}{2}e^{-\gamma\rho},\label{eq:pfrak0}\\
&\pfrak_1(\xt,\rho)=\frac{\alpha^2-\beta^2}{\gamma^2}\frac{\pi_t-\pi_c}{2}\frac{\afrak^0_{33}(\xt)}{2}\rho
 e^{-\gamma\rho}\label{eq:pfrak1},
 \end{align}
 \end{subequations}
while $\tilde{P}$ satisfies~\eqref{eq:pbPtilde}.
 \subsection{Numerical simulations}
 To illustrate our results, we perfom simulation in the simple spherical case, with constant Dirichlet conditions on the boundary of the unit sphere. By symmetry, the problem to solve reads as follows
 
 \begin{subequations}\label{eq:pbNum}
\begin{align}
&\frac{d^2}{dr^2}q^\eps+\frac{2}{r}\frac{d}{dr}q^\eps -\frac{\alpha^2+\beta^2}{\eps^2}q^\eps=0,&&\text{for $0<r<1$},
\\
&\frac{d^2}{dr^2}p^\eps+\frac{2}{r}\frac{d}{dr}p^\eps =\frac{\alpha^2-\beta^2}{\eps^2}q^\eps,&&\text{for $0<r<1$,}\label{eq:pbNump}
\intertext{with the boundary conditions}
&q^\eps|_{r=1}=\frac{\pi_t-\pi_c}{2}, \quad p^\eps|_{r=1}=\frac{\pi_t+\pi_c}{2},\qquad \frac{dq^\eps}{dr}|_{r=0}=\frac{dp^\eps}{dr}|_{r=0}=0.
\end{align}
\end{subequations}
Denote by $\gamma=\sqrt{\alpha^2+\beta^2}$. Note that $q^\eps$ admits the explicit formula 
$$q^\eps=\frac{\pi_t-\pi_c}{2}\frac{\sinh(\gamma r/\eps)}{r\sinh(\gamma \eps)}.$$
On the contrary,  the function $p^\eps$ does not have explicit solution and it requires a numerical method to be solved. We used here the full second order finite difference method, by discretizing the flux at the order 2 with a decentered stencil in $r=0$, and by approaching at the order 2 the operator $\dfrac{d^2}{dr^2} +\dfrac{1}{r}\dfrac{d}{dr}$ at any point $r_i$ of the discretization of the segment $(0,1)$. To avoid numerical instabilities near $r=0$ we multiply equation~\eqref{eq:pbNump} by $r$ and use the following approximation:
 $$r_i(\frac{d^2}{dr^2}f +\frac{1}{r}\frac{d}{dr}f)|_{r_i}\sim \frac{r_{i+1}f_{i+1}-2r_i f_i+r_{i-1}f_{i-1}}{\delta r^2}+O(\delta r^2),$$
 where $\delta r$ is the path of the discretization.
 It is worth noting that due to the exponential decay in $\gamma/\eps$ of $q^\eps$, solving numerically the problem satisfied by $p^\eps$ requires a very fine discretization especially the border $r=1$, as shown in Figure~\ref{fig:profilp}.
 \begin{figure}[h!]
\centering 
\includegraphics[width=.5\linewidth]{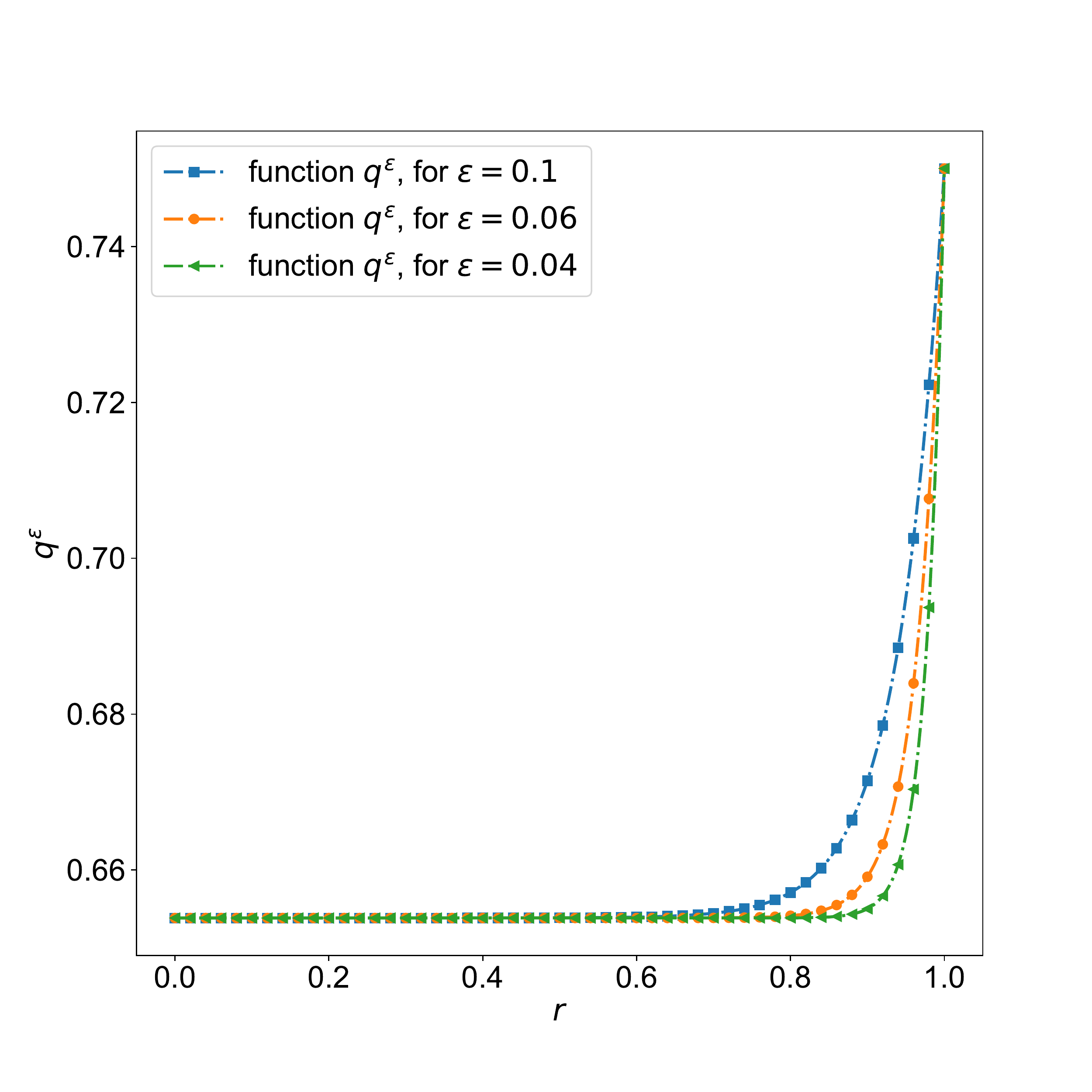}\hfill
\includegraphics[width=.5\linewidth]{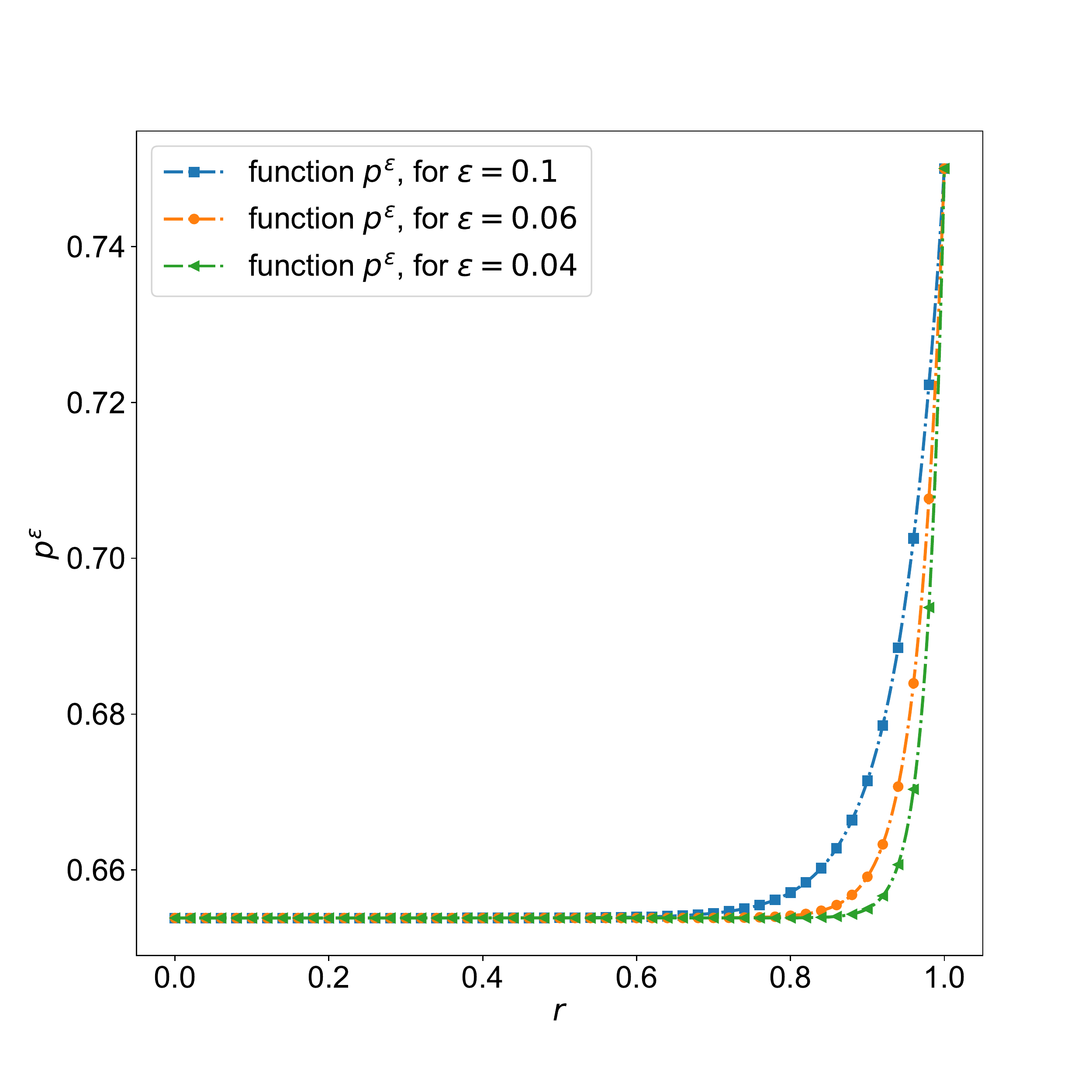}\hfill
\caption{Plots of the pressures  $q^\eps$ (Left) and $p^\eps$ (Right) for 3 values of $\varepsilon$: 0.1, 0.07 and 0.04. As $\eps$ goes to zero, a boundary layer appears near the boundary, making the numerical computation of the full problem costly and  unstable.
The following constants have been chosen: $\pi_t=0.5$, $\pi_c=1$, $\alpha=1$, $\beta=1.5$, and the discretization step is $\delta r=10^{-4}$. }
\label{fig:profilp}
\end{figure}
 
 On the other hand,  the solution to \eqref{eq:pbPtilde}  and the first order profiles given by~\eqref{eq:firstorder} provides the  approximation at the order $\eps^2$ of both $q^\eps$ and $p^\eps$ by
 denoting by $q_{approx,1}$ and $p_{approx,1}$ respectively the following functions:
\begin{subequations}
 \begin{align}
 q_{approx,1}(r)&=\frac{(\pi_t-\pi_c)(\xt)}{2}e^{-\gamma(1-r)/\eps}+ \frac{(\pi_t-\pi_c)(\xt)}{4}(1-r)e^{-\gamma(1-r)/\eps},\qquad \forall r\in(0,1)\\
 p_{approx,1}(r)&=\frac{\pi_t+\pi_c}{2}-\frac{\alpha^2-\beta^2}{\alpha^2+\beta^2}\left(\frac{\pi_t-\pi_c}{2}+q_{approx,1}\right),\qquad \forall r\in(0,1).
 \end{align}
 \end{subequations}
Figure~\ref{Fig:exact_num_comparisons} illustrates the order of convergence of Propositions~\ref{prop:estimexpan}--\ref{prop:peps}.

\begin{figure}[h!]
\centering 
\includegraphics[width=.5\linewidth]{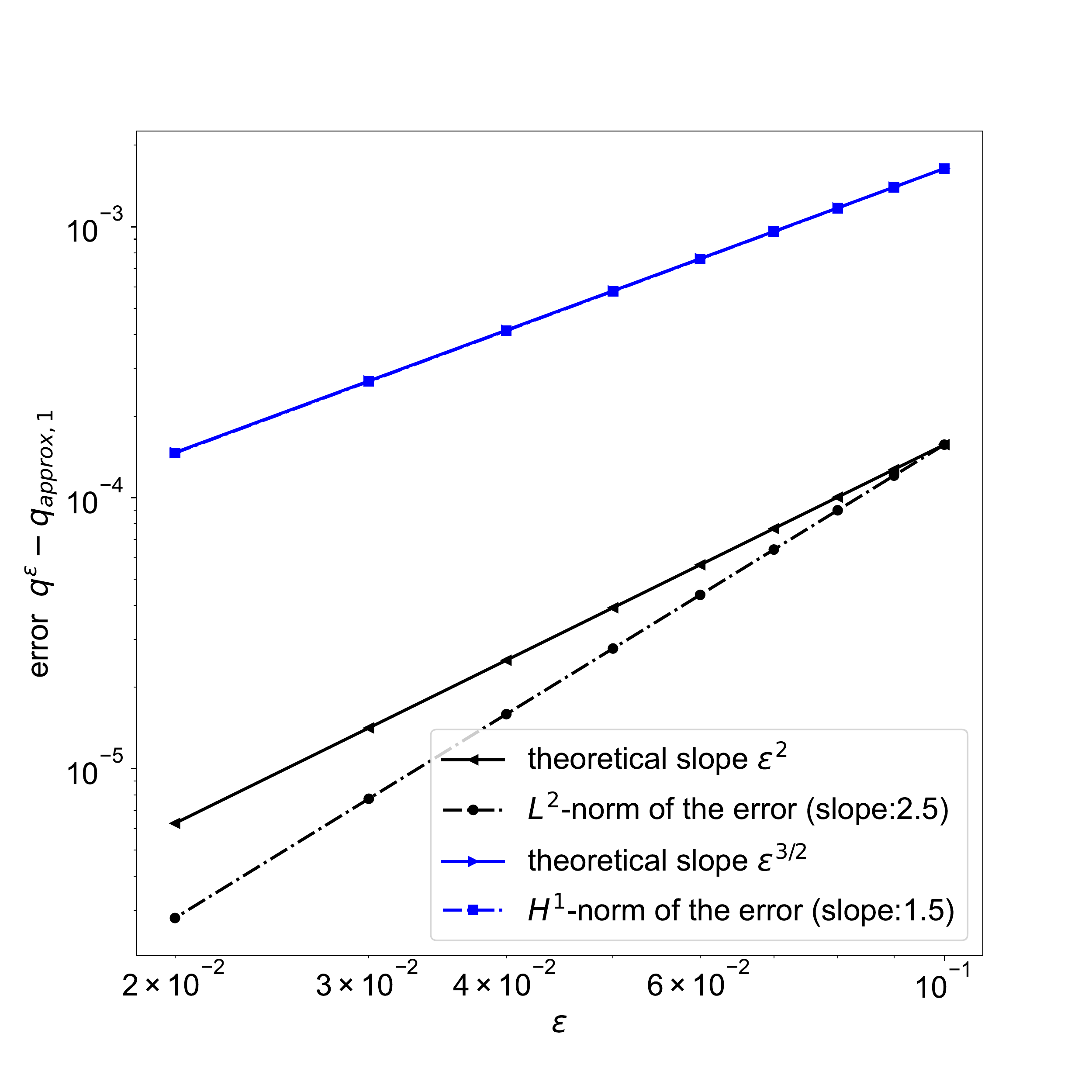}\hfill
\includegraphics[width=.5\linewidth]{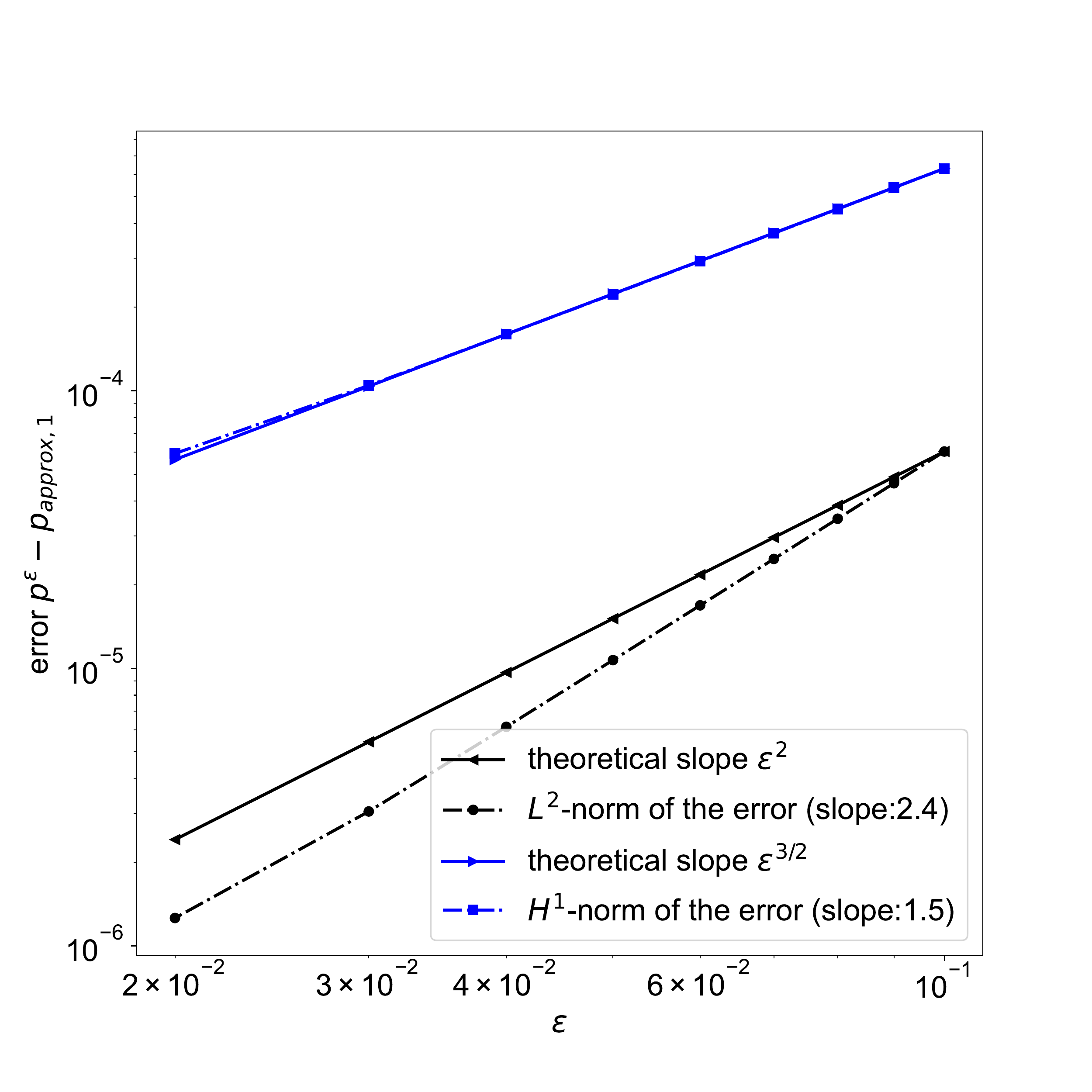}
\caption{(Left): Convergence rate of the L$^2$-norm and H$^1$-norm of $q^\eps-q_{approx,1}$. (Right): Convergence rate of the L$^2$-norm and H$^1$-norm of $p^\eps-p_{approx,1}$.
The following constants have been chosen: $\pi_t=0.5$, $\pi_c=1$, $\alpha=1$, $\beta=1.5$. In this simple particular case, the numerical slopes of the errors are slightly better than the theoretical results for the L$^2$ norm.}
\label{Fig:exact_num_comparisons}
\end{figure}

\end{document}